\documentclass[12pt,a4paper]{article}

\usepackage{amsmath,amssymb,amsthm,amsfonts,latexsym,graphicx,subfigure}
\usepackage[parfill]{parskip}
\usepackage[utf8]{inputenc}
\usepackage{xcolor}
\usepackage{float}
\usepackage{enumerate}
\usepackage[margin=3cm]{geometry}
\usepackage{url}

\newtheorem{theorem}{Theorem}[section]
\newtheorem{corollary}[theorem]{Corollary}

\newtheorem{lemma}[theorem]{Lemma}
\newtheorem{proposition}[theorem]{Proposition}
\newtheorem{remark}[theorem]{Remark}

\theoremstyle{definition}

\title{On the $k$-independence number of graph products}

\author{Aida Abiad\thanks{\texttt{a.abiad.monge@tue.nl}, Department of Mathematics and Computer Science, Eindhoven University of Technology, The Netherlands\newline Department of Mathematics: Analysis, Logic and Discrete Mathematics, Ghent University, Belgium\newline Department of Mathematics and Data Science of Vrije Universiteit Brussel, Belgium}\quad \qquad Hidde Koerts \thanks{\texttt{hkoerts@uwaterloo.ca}, Department of Combinatorics and Optimization, University of Waterloo, Canada}
}

\begin{document}

\date{} 
\maketitle

\begin{abstract}
The $k$-independence number of a graph, $\alpha_k(G)$, is the maximum size of a set of vertices at pairwise distance greater than $k$, or alternatively, the independence number of the $k$-th power graph $G^k$. Although it is known that $\alpha_k(G)=\alpha(G^k)$, this, in general, does not hold for most graph products, and thus  the existing bounds for $\alpha$ of graph products cannot be used. In this paper we present sharp upper bounds for the $k$-independence number of several graph products. In particular, we focus on the Cartesian, tensor, strong, and lexicographic products. Some of the bounds previously known in the literature for $k=1$ follow as corollaries of our main results.
\end{abstract}

\section{Introduction}

Consider two graphs $G_1, G_2$ and $k \geq 1$. A vertex set $S \subseteq V(G)$ is said to
be $k$\nobreakdash-\emph{independent} if $u, v \in S$ implies $\delta_G(u, v) > k$ where $\delta_G(u, v)$ is the shortest
distance between vertices $u$ and $v$ in graph $G$. The $k$-\emph{independence number} of graph $G$, denoted by $\alpha_k(G)$, is the size of
the largest $k$-independent vertex set in graph $G$. When $k = 1$ this reduces to the standard definition of independence number. We note that several conflicting definitions of the $k$-independence number are used in existing literature, all generalizing the concept of the independence number, for an overview see \cite[Section 1]{Abiad2019OnGraphs}.

The $k$-\emph{th graph power} $G^k$ of a graph $G$ is the graph whose vertex set is $V(G)$ in which two distinct vertices are adjacent if and only if their distance in graph $G$ is at most $k$. The $k$-independence number is equivalently defined as the independence number of the power graph, that is, $\alpha_k(G) = \alpha(G^k)$. However, even the simplest algebraic or combinatorial parameters of power graph $G^k$ cannot be deduced easily from the corresponding parameters of the graph $G$.  For instance, in general neither the spectrum~\cite{Das2013LaplacianGraph},~\cite[Section 2]{Abiad2022}, nor the average degree~\cite{Devos2013AveragePowers}, nor the rainbow connection number~\cite{Basavaraju2014RainbowProducts} can be derived directly from the original graph. This provides the main initial motivation for this work. 

The $k$-independence number of a graph has received a considerable amount of attention over the last years. From the complexity point of view, Kong and Zhao~\cite{Kong1993OnSets}, who showed that for every $k\geq 2$, determining $\alpha_k(G)$ is NP-complete for general graphs, and it remains NP-complete when restricting to regular bipartite graphs \cite{kz2000}. There are several other algorithmic results on $\alpha_k$, see for instance the work by Duckworth and Zito~\cite{Duckworth2003LargeGraphs} or Hota, Pal and Pal~\cite{HPP2001}. Since the $k$\nobreakdash-independence number is an NP-hard parameter, it is desirable to obtain sharp upper bounds. In this regard, Firby and Haviland~\cite{Firby1997IndependenceGraphs} proved an upper bound for $\alpha_k(G)$ in terms of the average distance in an $n$-vertex connected graph. Li and Wu~\cite{lw2021} showed sharp upper bounds on $\alpha_k$ for $t$-connected graphs. The $k$-independence number has also been studied from an algebraic point of view by Abiad~et~al.~\cite{Abiad2016SpectralGraph,Abiad2019OnGraphs,Abiad2022}, Fiol~\cite{Fiol2020ANumber} and O~et~al.~\cite{Taoqiu2019SharpGraphs}. Wocjan~et~al.~\cite{Wocjan2019SpectralGraph} have shown bounds on the quantum $k$-independence number, a related parameter which is used to measure the benefit of quantum entanglement. For each fixed integer $k \geq 2$ and $r\geq 3$, Beis, Duckworth, and Zito~\cite{bdz2005} proved several upper bounds for $\alpha_k(G)$ in random $r$-regular graphs. The $k$-independence number has also been studied in the context of the random graph $G_{n,p}$ by Atkinson and Frieze~\cite{AF2003}.


In this paper we show several sharp upper bounds for the $k$-independence number of graph products. For a pair of graphs $G_1, G_2$, we consider the Cartesian product, the tensor product, the strong product, and the lexicographic product, denoted by $G_1 \square G_2$, $G_1 \times G_2$, $G_1 \boxtimes G_2$, and $G_1 \cdot G_2$, respectively. We note that the tensor product is also known as the direct product, the Kronecker product, and the categorical product. The vertex set for all these product graphs is given by the Cartesian product of the vertex sets $V_1$ and $V_2$. The edge sets of the product graphs are given as follows:
	{\footnotesize
	\begin{align*}
		& E(G_1 \square G_2) = \{((u_1, u_2), (v_1, v_2))\, |\, (u_1 = v_1 \land (u_2, v_2) \in E(G_2)) \lor (u_2 = v_2 \land (u_1, v_1) \in E(G_1))\}\\
		& E(G_1 \times G_2) = \{((u_1, u_2), (v_1, v_2))\, |\, (u_1, v_1) \in E(G_1) \land (u_2, v_2) \in E(G_2)\}\\
		& E(G_1 \boxtimes G_2) = E(G_1 \square G_2) \cup E(G_1 \times G_2)\\
		& E(G_1 \cdot G_2) = \{((u_1, u_2), (v_1, v_2)) , | \, (u_1, v_1) \in E(G_1) \lor (u_1 = v_1 \land (u_2, v_2) \in E(G_2))\}
	\end{align*}}
There are several well-known results for the independence number of graph products. Instances of it are the work of Vizing~\cite{Vizing1963CartesianGraphs}, Sonnemann and~Krafft~\cite{Sonnemann1974IndependenceGraphs}, Jha and~Slutzki~\cite{Jha1994IndependenceGraphs}, Klav\v{z}ar~\cite{Klavzar2005SomeGraphs}, Jha and~Klav\v{z}ar~\cite{Jha1998IndependenceID}, Geller and Stahl~\cite{GELLER197587}, and \v{S}pacapan~\cite{SPACAPAN20111377}, among others. Although $\alpha_k(G)=\alpha(G^k)$, in general it does not hold that the $k$-independence number of the product of two graphs is equivalent to the independence number of the product of the corresponding two graph powers (in fact,  this only holds for the strong product out of the four considered graph products). In this paper we provide new tight bounds for the $k$-independence number of the most well known graph products: the strong product (Section~\ref{sec:strongproduct}), the Cartesian product (Section~\ref{sec:cartesianproduct}), the tensor product (Section~\ref{sec:tensorproduct}), and the lexicographic product (Section~\ref{sec:lexicographicproduct}). 
Some of the bounds previously known in the literature for $k=1$ follow as corollaries of our main results.

\section{Strong product}\label{sec:strongproduct}

In this section we will show that one can use the equivalence $\alpha_k(G)=\alpha(G^k)$ to upper bound the independence number of the strong product of two graphs. To that aim, we first need some preliminary results.

\begin{proposition}
	\label{prop:strongproductgraphdistance}
	{\cite[Proposition~5.4]{handbook}}
	For any two graphs $G_1, G_2$, it holds for any pair of vertices $(u_1, u_2), (v_1, v_2) \in V(G_1) \times V(G_2)$ in the product graph $G_1 \boxtimes G_2$ that \[\delta_{G_1 \boxtimes G_2} ((u_1, u_2), (v_1, v_2)) = \max (\delta_{G_1}(u_1, v_1), \delta_{G_2}(u_2, v_2)).\] 
\end{proposition}

For the strong product of two graphs, the next result shows that one can use existing bounds for the independence number on power graphs.

\begin{lemma}
	\label{lem:strongproductpowerequivalence}
	For any graphs $G_1, G_2$, it holds that $G_1^k \boxtimes G_2^k = (G_1 \boxtimes G_2)^k$.
\end{lemma}
\begin{proof}
	Let the graphs $G_1$ and $G_2$ be given. We first observe that 
	\begin{align*}
	V(G_1^k \boxtimes G_2^k) &= V(G_1^k) \times V(G_2^k)\\
	&= V(G_1) \times V(G_2)\\
	&= V(G_1 \boxtimes G_2)\\
	&= V((G_1 \boxtimes G_2)^k).
	\end{align*}
	It thus remains to show that $E(G_1^k \boxtimes G_2^k) = E((G_1 \boxtimes G_2)^k)$. Let $(u_1, u_2)$ and $(v_1, v_2)$ be two distinct elements in the set $V(G_1) \times V(G_2)$. By the definition of the strong graph product, we note that $((u_1, u_2), (v_1, v_2))$ forms an edge in graph $G_1^k \boxtimes G_2^k$ if and only if $(u_1, v_1) \in E(G_1^k)$ or $u_1 = v_1$, and $(u_2, v_2) \in E(G_2^k)$ or $u_2 = v_2$. By the definition of graph powers, this holds if and only if
	$$\delta_{G_1}(u_1, v_1), \delta_{G_2}(u_2, v_2) \leq k.$$ 
	We note that this expression is in turn equivalent to the inequality 
	$$\max(\delta_{G_1}(u_1, v_1), \delta_{G_2}(u_2, v_2)) \leq k.$$ 
	By Proposition~\ref{prop:strongproductgraphdistance}, this inequality is then equivalent to 
	$$\delta_{G_1 \boxtimes G_2}((u_1, u_2), (v_1, v_2)) \leq k.$$ 
	Finally, we observe that by the definition of graph powers, the inequality \linebreak $\delta_{G_1 \boxtimes G_2}((u_1, u_2), (v_1, v_2)) \leq k$ holds if and only if $((u_1, u_2), (v_1, v_2)) \in E((G_1 \boxtimes G_2)^k)$. Thus, $((u_1, u_2), (v_1, v_2)) \in E(G_1^k \boxtimes G_2^k)$ if and only if $((u_1, u_2), (v_1, v_2)) \in E((G_1 \boxtimes G_2)^k)$. Therefore, $E(G_1^k \boxtimes G_2^k) = E((G_1 \boxtimes G_2)^k)$, as desired.
\end{proof}

\begin{corollary}\label{cor:strongproductrelation}
	For any graphs $G_1, G_2$, it holds that $\alpha (G_1^k \boxtimes G_2^k) = \alpha_k(G_1 \boxtimes G_2)$.
\end{corollary}

Note that Corollary \ref{cor:strongproductrelation} implies that the existing upper bounds for $\alpha$ for the strong product of two graphs (see for instance Jha and~Slutzki~\cite[Theorem~2.6]{Jha1994IndependenceGraphs}) can be used.
As an application, one can easily extend \cite[Theorem~2.6]{Jha1994IndependenceGraphs} to the $k$-independence number $\alpha_k$.

\begin{theorem}
	\label{thm:product_graphs_strong_lower_bound}
	For all graphs $G_1, G_2$, \[\alpha_k(G_1 \boxtimes G_2) \geq \alpha_k(G_1) \cdot \alpha_k(G_2).\]
\end{theorem}
\begin{proof}
	It follows directly from Corollary~\ref{cor:strongproductrelation} and applying~\cite[Theorem~2.6]{Jha1994IndependenceGraphs} to the power graphs $G_1^k$ and $G_2^k$.
\end{proof}

The relation given by Lemma~\ref{lem:strongproductpowerequivalence} does not extend to the other graph products considered in this paper. For $k = 2$, take for instance  $G_1$ to be a complete graph $K_2$, and $G_2$ to be a path $P_4$; then $G_1^k \square G_2^k \neq (G_1 \square G_2)^k$, $G_1^k \times G_2^k \neq (G_1 \times G_2)^k$, and $G_1^k \cdot G_2^k \neq (G_1 \cdot G_2)^k$.

\section{Cartesian product}\label{sec:cartesianproduct}

Vizing~\cite{Vizing1963CartesianGraphs} obtained the following celebrated bounds on $\alpha(G_1\square G_2)$:

\begin{theorem}\cite{Vizing1963CartesianGraphs}\label{thm:Vizing}
	For any two graphs $G_1, G_2$,
	\begin{description}
		\item[$(i)$] $\alpha(G_1 \square G_2) \geq \alpha(G_1) \cdot \alpha(G_2) + \min(|V(G_1)| - \alpha(G_1), |V(G_2)| - \alpha(G_2)),$
		\item[$(ii)$] $\alpha(G_1 \square G_2) \leq \min(\alpha(G_1) \cdot |V(G_2)|, \alpha(G_2) \cdot |V(G_1)|).$
	\end{description}
\end{theorem}
It is easy to see that if both $G_1$ and $G_2$ are complete graphs, then Theorem~\ref{thm:Vizing} yields the exact value of $\alpha(G_1\square G_2)$. However, in general, there is a gap between the two bounds. 

In this section we will extend Vizing's bounds to the $k$-independence number. To that purpose we will use the relation between distances in graphs and distances in their graph products. Recall that $\delta_G(v_i, v_j)$ denotes the distance between two vertices $v_i, v_j$ in a graph $G$.

\begin{proposition} 
	\label{prop:cartesianproductgraphdistance}
	{\cite[Proposition~5.1]{handbook}}
	For any two graphs $G_1, G_2$, it holds for any pair of vertices $(u_1, u_2), (v_1, v_2) \in V(G_1) \times V(G_2)$ in the product graph $G_1 \square G_2$ that 
	\[\delta_{G_1 \square G_2} ((u_1, u_2), (v_1, v_2)) = \delta_{G_1}(u_1, v_1) + \delta_{G_2}(u_2, v_2).\]
\end{proposition}

The following two results extend Vizing's lower and upper bounds from Theorem~\ref{thm:Vizing}.

\begin{theorem}    \label{thm:product_graphs_cartesian_bounds}
	For any two graphs $G_1, G_2$,
	\begin{description}
		\item[$(i)$] $\alpha_k(G_1 \square G_2) \geq \alpha_k(G_1) \cdot \alpha_k(G_2),$
		\item[$(ii)$] $\alpha_k(G_1 \square G_2) \leq \min(\alpha_k(G_1) \cdot |V(G_2)|, \alpha_k(G_2) \cdot |V(G_1)|).$
	\end{description}
\end{theorem}

\begin{proof}
	\begin{description}
		\item[$(i)$]   
		Let the graphs $G_1$ and $G_2$ and $k \in \mathbb{N}$ be given. Let $S_1 \subseteq V(G_1)$ be a set of vertices in graph $G_1$ such that $|S_1| = \alpha_k(G_1)$ and $S_1$ is $k$-independent in the graph $G_1$. Similarly, let $S_2 \subseteq V(G_2)$ be a set of vertices in graph $G_2$ such that $|S_2| = \alpha_k(G_2)$ and $S_2$ is $k$-independent in the graph $G_2$. We claim that the set of vertices $S = S_1 \times S_2$ is $k$-independent in the product graph $G_1 \square G_2$. 
		
		Let $(u_1, u_2), (v_1, v_2) \in S$ be two distinct vertices in the product graph $G_1 \square G_2$. By Proposition~\ref{prop:cartesianproductgraphdistance}, 
		\[\delta_{G_1 \square G_2} ((u_1, u_2), (v_1, v_2)) = \delta_{G_1}(u_1, v_1) + \delta_{G_2}(u_2, v_2).\]
		Because the vertices $(u_1, u_2)$ and $(v_1, v_2)$ are distinct, either $u_1 \neq v_1$ or $u_2 \neq v_2$ must hold. Without loss of generality, assume that $u_1 \neq v_1$. In that case, it follows that $\delta_{G_1 \square G_2} ((u_1, u_2), (v_1, v_2)) \geq \delta_{G_1}(u_1, v_1)$. As $u_1, v_1 \in S_1$ and because set $S_1$ is $k$-independent in graph $G_1$, it holds that $\delta_{G_1}(u_1, v_1) > k$. Therefore it follows that $$\delta_{G_1 \square G_2} ((u_1, u_2), (v_1, v_2)) > k.$$But as vertices $(u_1, u_2), (v_1, v_2) \in S$ were selected arbitrarily, we conclude that the set $S$ is $k$-independent in the product graph $G_1 \square G_2$. As $|S| = |S_1| \cdot |S_2| = \alpha_k(G_1) \cdot \alpha_k(G_2)$, we then conclude that $$\alpha_k(G_1 \square G_2) \geq \alpha_k(G_1) \cdot \alpha_k(G_2).$$ 
		\item[$(ii)$] 
		Consider two graphs $G_1$ and $G_2$ and $k \in \mathbb{N}$. Let $S \subseteq V(G_1)$ be a set of vertices in graph $G_1$ such that $|S| = \alpha_k(G_1)$ and set $S$ $k$-independent. Let $G_2'$ be the edgeless graph on the vertex set of graph $G_2$, that is, $V(G_2') = V(G_2)$ and $E(G_2') = \emptyset$. Then, as removing edges from a graph cannot result in a decrease of the $k$-independence number, by the definition of the Cartesian graph product, $\alpha_k(G_1 \square G_2) \leq \alpha_k(G_1 \square G_2')$.
		
		In the graph product $G_1 \square G_2'$, due to the non-existence of edges in $G_2'$ and the definition of the Cartesian graph product, we know that the sets $S_{v_2} = \{(v_1, v_2) | v_1 \in V(G_1)\}$ for vertex ${v_2 \in V(G_2)}$ are anticomplete. Let the graph $G_{v_2}$ be the subgraph induced by set $S_{v_2}$ in product graph $G_1 \square G_2'$ for \linebreak${v_2 \in V(G_2)}$. Then, as the sets $S_{v_2}$ for all vertices $v_2 \in V(G_2)$ are anticomplete, $\alpha_k(G_1 \square G_2') = \sum_{v_2 \in V_2} \alpha_k(G_{v_2})$. Moreover, as $((x, v_2), (y, v_2)) \in E(G_1 \square G_2)$ if and only if $(x, y) \in E(G_1)$, $\alpha_k(G_{v_2}) = \alpha_k{G_1}$ for all vertices $v_2 \in V(G_2)$. Therefore, $\alpha_k(G_1 \square G_2') = \alpha_k(G_1) \cdot |V_2|$.
		
		Thus, $\alpha_k(G_1 \square G_2) \leq \alpha_k(G_1 \square G_2') = \alpha_k(G_1) \cdot |V_2|$. Analogously, it follows that $\alpha_k(G_1 \square G_2) \leq \alpha_k(G_2) \cdot |V_1|$.
		\qedhere
	\end{description}
\end{proof}

While Theorem~\ref{thm:product_graphs_cartesian_bounds}$(ii)$ is tight for complete graphs when $k=1$~\cite{Vizing1963CartesianGraphs}, this is not the case for $k > 1$. Indeed, take for example $G_1 = G_2 = K_2$, and $k = 2$. Then $\alpha_k(G_1 \square G_2) = 1$, while $\alpha_k(G_1) \cdot |V_2| = \alpha_k(G_2) \cdot |V_1| = 2$. On the other hand, Theorem~\ref{thm:product_graphs_cartesian_bounds}$(i)$ is tight for $G_1, G_2$ being complete graphs and $k > 1$, as the distance between any pair of vertices in the graph product $G_1 \square G_2$ is then at most $2$. Thus, for $k > 1$, $\alpha_k(G_1 \square G_2) = 1 = 1 \cdot 1 = \alpha_k(G_1) \cdot \alpha_k(G_2)$. Observe that if the graph $G_1$ or $G_2$ are edgeless, then the upper and lower bound coincide and are thus tight.

For $k=1$, Theorem~\ref{thm:product_graphs_cartesian_bounds} yields Vizing's bounds from Theorem~\ref{thm:Vizing} (see~\cite{Vizing1963CartesianGraphs} for more details) and Jha and Slutzki~[\cite{Jha1994IndependenceGraphs}, Corollary~2.5]. Next, we investigate other tight cases of Theorem~\ref{thm:product_graphs_cartesian_bounds}. 

\begin{remark}
	\begin{description}
		\item[$(i)$] Theorem~\ref{thm:product_graphs_cartesian_bounds}$(i)$ is tight for all even $k \in \mathbb{N}$ for graphs $G_1$ and $G_2$ both isomorphic to the path  $P_{k+2}$.
		\item[$(ii)$] Theorem~\ref{thm:product_graphs_cartesian_bounds}$(i)$ is tight for two graphs $G_1$ and $G_2$ if and only if there exists a maximum $k$-independent set $S$ in the graph product  $G_1\square G_2$ such that for all the vertices $(u_1, v_1)$ and $(u_2, v_2)$ in set $S$, it holds that vertices $(u_1, v_2)$ and $(u_2, v_1)$ are contained in the set $S$ as well.
		\item[$(iii)$] Theorem~\ref{thm:product_graphs_cartesian_bounds}$(ii)$ is tight for all $k \in \mathbb{N}^+$, if the graph $G_1$ is $K_2$, and $G_2$ is $C_{2k+1}$.
	\end{description}
\end{remark}
\begin{proof}
	\begin{description}
		\item[$(i)$] Let the vertices of the paths $G_1$ and $G_2$ be labelled by $u_1, u_2, \ldots , u_{k+2}$ and $v_1, v_2, \ldots , v_{k+2}$, respectively, and such that $(u_i, u_{i+1}) \in E(G_1)$ and $(v_i, v_{i+1}) \in E(G_2)$ for $i = 1, 2, \ldots , k+1$. Trivially, $\alpha_k(G_1) = \alpha_k(G_2) = 2$. Thus, it suffices to show that $\alpha_k(G_1 \square G_2) \leq 4$. Let $S$ be a maximum $k$-independent set of vertices in $G_1 \square G_2$. We aim to show that $|S| \leq 4$. To that purpose, we consider a partitioning of the vertex set $V(G_1 \square G_2)$ given by the sets 
		\begin{align*}
		V_1 &= \{(u_i, v_j) \, | \, 1 \leq i \leq \frac{k}{2}+1, \, 1 \leq j \leq \frac{k}{2}+1\},\\
		V_2 &= \{(u_i, v_j) \, | \, 1 \leq i \leq \frac{k}{2}+1, \, \frac{k}{2}+2 \leq j \leq k+2\},\\
		V_3 &= \{(u_i, v_j) \, | \, \frac{k}{2}+2 \leq i \leq k+2, \, 1 \leq j \leq \frac{k+2}{2}\},\\
		V_4 &= \{(u_i, v_j) \, | \, \frac{k}{2}+2 \leq i \leq k+2, \, \frac{k}{2}+2 \leq j \leq k+2\}.
		\end{align*}
		We note that as $k$ is even, the partitioning is well-defined. Next, we observe that for any pair of vertices $(u_{i_1}, v_{j_1}), (u_{i_2}, v_{j_2}) \in V_1$, by Proposition~\ref{prop:cartesianproductgraphdistance}, it follows that 
		\begin{align*}
		\delta_{G_1 \square G_2}((u_{i_1}, v_{j_1}), (u_{i_2}, v_{j_2})) &= \delta_{G_1}(u_{i_1}, u_{i_2}) + \delta_{G_2}(v_{j_1}, v_{j_2})\\
		&= |i_2 - i_1| + |j_2 - j_1|\\
		&\leq \frac{k}{2} + \frac{k}{2}\\
		&= k.
		\end{align*}
		Thus, a maximum $k$-independent set $S$ contains at most one vertex in set $V_1$. Analogously, it follows that $S$ contains at most one vertex of each of the sets $V_2$, $V_3$, and $V_4$. Then, as the sets $V_1$, $V_2$, $V_3$, and $V_4$ partition the vertex set $V(G_1 \square G_2)$, we find that $S$ contains at most four vertices, as desired.		
				
		\item[$(ii)$] Let $S_1$ and $S_2$ be maximum $k$-independent sets in $G_1$ and $G_2$, respectively. As shown in the proof of Theorem~\ref{thm:product_graphs_cartesian_bounds}$(i)$, the set $S = S_1 \times S_2$ is a $k$-independent set in the graph product $G_1 \square G_2$. If the bound is tight for $G_1$ and $G_2$, the set $S$ must be a maximum $k$-independent set in $G_1 \square G_2$. Because $S$ is the Cartesian product of two sets, it trivially holds that for all vertices $(u_1, v_1)$ and $(u_2, v_2)$ in  $S$, that vertices $(u_1, v_2)$ and $(u_2, v_1)$ are in $S$ as well.
		
		Next consider a set $S\subseteq V(G_1) \times V(G_2)$ such that $S$ is a maximum $k$-independent set in the product $G_1 \square G_2$ such that for all vertices $(u_1, v_1)$ and $(u_2, v_2)$ in set $S$, vertices $(u_1, v_2)$ and $(u_2, v_1)$ are elements of set $S$ as well. Then, clearly, there must exist sets $S_1 \subseteq V_1$ and $S_2 \subseteq V_2$ such that $S_1 \times S_2 = S$. But then, by reversing the argument in the proof of Theorem~\ref{thm:product_graphs_cartesian_bounds}$(i)$, $S_1$ and $S_2$ must be $k$-independent in graphs $G_1$ and $G_2$ respectively. Hence, as $|S_1| \leq \alpha_k(G_1)$ and $|S_2| \leq \alpha_k(G_2)$, the bound must be attained tightly.
		
		Thus both directions of the bi-implication hold.
		
		\item[$(iii)$] It suffices to show that for all $k \in \mathbb{N}^+$ there exists a $k$-independent set $S \subseteq V(G)$ such that $|S| = \min(\alpha_k(K_2)\cdot |V(C_{2k+1})|, \alpha_k(C_{2k+1})\cdot |V(K_2)|)$. We observe that $\alpha_k(C_{2k+1}) = 1$, $\alpha_k(K_2) = 1$, $|V(C_{2k+1})| = 2k+1$, and $|V(K_2)|=2$. Hence, it suffices to show that for each $k \in \mathbb{N}^*$ there exists a $k$-independent set $S \subseteq V(G)$ such that $|S| = 2$.
		
		Let $V(K_2) = \{u_1, u_2\}$ and let $V(C_{2k+1}) = \{v_1, v_2, \ldots , v_{2k+1}\}$ such that \linebreak $E(C_{2k+1}) = \{(v_i,v_{i+1})\, | \, i \in [2k]\} \cup \{(v_1, v_{2k+1})\}$. Then we claim that $S = \{(u_1,v_1), (u_2, v_{k+1})\}$ is $k$-independent in the graph product $K_2 \square C_{2k+1}$. It suffices to show that $\delta_{K_2 \square C_{2k+1}}((u_1, v_1), (u_2, v_{k+1})) > k$. From Proposition~\ref{prop:cartesianproductgraphdistance}, it follows that
		\begin{align*}
		\delta_{K_2 \square C_{2k+1}}((u_1, v_1), (u_2, v_{k+1})) &= \delta_{K_2}(u_1, u_2) + \delta_{C_{2k+1}}(v_1, v_{k+1})\\
		&= 1 + k\\
		&> k,
		\end{align*}
		as desired.
		\qedhere
	\end{description}

\end{proof}

\section{Tensor product}\label{sec:tensorproduct}
We first note that the tensor graph product does not preserve the connectedness of the original graphs (take for instance the graph product $P_2 \times P_3$). Moreover, the distance between two vertices in the graph product can generally not be derived directly from the distances between the corresponding vertices in the graphs $G_1$ and $G_2$. 

\begin{proposition}
	\label{prop:tensorproductgraphdistance2}
	{\cite[Proposition~5.7]{handbook}}
	For any two graphs $G_1, G_2$, it holds that for any pair of vertices $(u_1, u_2), (v_1, v_2) \in V(G_1) \times V(G_2)$ in the graph product $G_1 \times G_2$ that $\delta_{G_1 \times G_2}((u_1, u_2), (v_1, v_2))$ equals the minimum number $\ell \in \mathbb{N}$ such that there exists both a walk between vertices $u_1$ and $v_1$ in graph $G_1$ of length $\ell$ and a walk between vertices $u_2$ and $v_2$ in graph $G_2$ of length $\ell$. If no such value $\ell \in \mathbb{N}$ exists, $\delta_{G_1 \times G_2}((u_1, u_2), (v_1, v_2)) = \infty$.
\end{proposition}
\begin{corollary}
	\label{cor:tensorproductgraphdistance3}
	For any two graphs $G_1, G_2$, it holds for any pair of vertices \linebreak$(u_1, u_2), (v_1, v_2) \in V(G_1) \times V(G_2)$ in the graph product $G_1 \times G_2$ that \[\delta_{G_1 \times G_2}((u_1, u_2), (v_1, v_2)) \geq \max (\delta_{G_1}(u_1, v_1), \delta_{G_2}(u_2, v_2)).\]
\end{corollary}

The next result extends a lower bound by Jha and~Slutzki~\cite[Theorem~2.4]{Jha1994IndependenceGraphs}.

\begin{theorem}
	\label{thm:product_graphs_tensor_lower_bound}
	For all graphs $G_1, G_2$, \[\alpha_k(G_1 \times G_2) \geq \alpha_k(G_1) \cdot \alpha_k(G_2).\]
\end{theorem}
\begin{proof}
	Consider two graphs $G_1$ and $G_2$ and $k \in \mathbb{N}$. Let $S_1 \subseteq V(G_1)$ and $S_2 \subseteq V(G_2)$ be sets of vertices such that $|S_1| = \alpha_k(G_1)$ and $|S_2| = \alpha_k(G_2)$, and sets $S_1$ and $S_2$ be $k$-independent in $G_1$ and $G_2$ respectively. Let $S \subseteq V(G_1) \times V(G_2)$ be the vertex set given by $S_1 \times S_2$.
	
	We will show that $S$ is $k$-independent in $G_1 \times G_2$. For two distinct vertices \linebreak $(u_1, u_2), (v_1, v_2) \in S$, we consider the distance $\delta_{G_1 \times G_2}((u_1, u_2), (v_1, v_2))$. Because vertices $(u_1, u_2)$ and $(v_1, v_2)$ are distinct, either $u_1 \neq v_1$ or $u_2 \neq v_2$ must hold. Without loss of generality, assume that $u_1 \neq v_1$. Then, because by Corollary~\ref{cor:tensorproductgraphdistance3}, $$\delta_{G_1 \times G_2} ((u_1, u_2), (v_1, v_2)) \geq \delta_{G_1} (u_1, v_1),$$ and because $u_1, v_1 \in S_1$ and set $S_1$ is $k$\nobreakdash-independent in graph $G_1$, $$\delta_{G_1 \times G_2} ((u_1, u_2), (v_1, v_2)) \geq \delta_{G_1} (u_1, v_1) > k.$$ Thus, $S$ is $k$\nobreakdash-independent in the product graph $G_1 \times G_2$. Since $|S| = \alpha_k(G_1) \cdot \alpha_k(G_2)$, we conclude that $\alpha_k(G_1 \times G_2) \geq \alpha_k(G_1) \cdot \alpha_k(G_2)$, as desired.
\end{proof}

We note that for $k=1$, the lower bound by Jha and Slutzki~\cite[Theorem~2.4]{Jha1994IndependenceGraphs} generally gives a better bound than Theorem~\ref{thm:product_graphs_tensor_lower_bound}. Next, we present a tight case of Theorem~\ref{thm:product_graphs_tensor_lower_bound}.

\begin{remark}
	Theorem~\ref{thm:product_graphs_tensor_lower_bound} is tight for all even $k \in \mathbb{N}$ and for $G_1$ and $G_2$ both being $P_{k+2}$.
\end{remark}
\begin{proof}
	Consider the paths $G_1$ and $G_2$ and label their vertices as $u_1, u_2, \ldots , u_{k+2}$ and $v_1, v_2, \ldots , v_{k+2}$ respectively, such that $(u_i, u_{i+1}) \in E(G_1)$ and $(v_i, v_{i+1}) \in E(G_2)$ for $i = 1, 2, \ldots , k+1$. Trivially, $\alpha_k(G_1) = \alpha_k(G_2) = 2$. Thus, it suffices to show that $\alpha_k(G_1 \times G_2) \leq 4$.
	
	We first observe that between any pair of vertices $u_i, u_j \in V(G_1)$ there exists only a single path in $G_1$, and analogously for any pair of vertices $v_i, v_j \in V(G_2)$ in $G_2$. Hence, by Proposition~\ref{prop:tensorproductgraphdistance2}, for any pair of vertices $(u_i, v_j), (u_{i'}, v_{j'}) \in V(G_1 \times G_2)$ in the product graph $G_1 \times G_2$ with $i,i', j,j' \in [k+2]$, it holds that \linebreak $\delta_{G_1 \times G_2} ((u_i, v_j), (u_{i'}, v_{j'})) \leq \ell$ if and only if $\delta_{G_1}(u_i, u_{i'})$ and $\delta_{G_2}(v_j, v_{j'})$ are both at most $\ell$ and of the same parity. Moreover, by the structure of path graphs $G_1$ and $G_2$, it holds that $\delta_{G_1}(u_i, u_{i'}) = |i - i'|$ and $\delta_{G_2}(v_j, v_{j'}) = |j - j'|$.
	
	By considering the maximum distance $\ell = |V(G_1 \times G_2)| = (k+2)^2$, we then note that the product graph $G_1 \times G_2$ consists of two disjoint connected components $C_1$ and $C_2$, induced by the sets $S_1 = \{(u_i, v_j) \in V(G_1 \times G_2) \, | \, i \equiv j \pmod 2\}$ and $S_2 = \{(u_i, v_j) \in V(G_1 \times G_2) \, | \, i \not\equiv j \pmod 2\}$ respectively. We moreover note that because $k$ is even, and as $G_1$ and $G_2$ are path graphs of even lengths, the components $C_1$ and $C_2$ are isomorphic. It thus suffices to show that $\alpha_k(C_1) \leq 2$.
	
	We consider the set $S_1' = \{(u_i, v_j) \in S_1 \, | \, i \equiv 0 \pmod 2\}$. Let $(u_i, v_j), (u_{i'}, v_{j'})$ be a pair of vertices in $S_1'$. By the definition of $S_1$, it follows that $|i -i'|$ and $|j -j'|$ have the same parity, and hence $\delta_{G_1 \times G_2}((u_i, v_j), (u_{i'}, v_{j'})) = \max(|i-i'|, |j-j'|)$. Next, by the definition of $S_1'$, it follows that $i,i',j,j' \neq 1$. Hence, $i,i',j,j' \in [2, k+2]$, and thus $\delta_{G_1 \times G_2}((u_i, v_j), (u_{i'}, v_{j'})) \leq k$. Thus, as $S_1'$ does not contain two vertices at distance greater than $k$, we conclude that the $k$-independence number of the subgraph of component $C_1$ induced by the set $S_1'$ is at most $1$.
	
	Analogously, it follows that the $k$-independence number of the subgraph of component $C_1$ induced by the set $S_1 \setminus S_1'$ is also at most $1$. Then, as $V(C_1) = S_1$, we find that $\alpha_k(C_1) \leq 2$, as desired.
\end{proof}

\section{Lexicographic product}\label{sec:lexicographicproduct}
The lexicographic product differs from the other graph products previously discussed in that it is non-commutative. That is, for graphs $G_1$ and $G_2$ it does generally not hold that $G_1 \cdot G_2 = G_2 \cdot G_1$. Moreover, as shown by Geller and Stahl~\cite{GELLER197587}, the independence number of a lexicographic product graph $G_1 \cdot G_2$ can directly be computed from the independence numbers of the factors $G_1$ and $G_2$:

\begin{theorem}
	\label{thm:independence_number_lexicographic_product}
	{\cite[Theorem~1]{GELLER197587}}
	For all graphs $G_1, G_2$, $\alpha(G_1 \cdot G_2) = \alpha(G_1)\cdot \alpha(G_2)$.
\end{theorem}

We note that while the lexicographic product is non-commutative, the order of the factors does not influence the independence number of the product graph. Moreover, we observe that it does not in general hold that $\alpha_k(G_1 \cdot G_2) = \alpha_k(G_1)\cdot\alpha_k(G_2)$. Consider for instance the case where $k = 2$, $G_1=K_2$, and $G_2=2K_1$. 

We investigate the $k$-independence number of lexicographic product graphs, and specifically the case where $k \geq 2$. To that aim, we need some preliminary results.

\begin{proposition}
	\label{prop:disjoint_union}
	Let $G$ be a graph and let $\mathcal{C}$ be a collection of graphs such that graph $G$ is the disjoint union of the graphs in $\mathcal{C}$. Then, for all $k \in \mathbb{N}^+$, $$\alpha_k(G) = \sum_{C \in \mathcal{C}} \alpha_k(C).$$
\end{proposition}



\begin{proposition}
	\label{prop:lexicographicproductgraphdistance}
	{\cite[Proposition~5.12]{handbook}}
	For any two graphs $G_1, G_2$, it holds for any pair of vertices $(u_1, u_2), (v_1, v_2) \in V(G_1) \times V(G_2)$ in the product graph $G_1 \cdot G_2$ that \[\delta_{G_1 \cdot G_2} ((u_1, u_2), (v_1, v_2)) = \begin{cases}
	\delta_{G_1}(u_1, v_1) & \text{ if } u_1 \neq v_1\\
	\delta_{G_2}(u_2, v_2) & \text{ if } u_1 = v_1 \text{ and } \deg_{G_1}(u_1) = 0\\
	\min(\delta_{G_2}(u_2, v_2), 2) & \text{ if } u_1 = v_1 \text{ and } \deg_{G_1}(u_1) > 0.
	\end{cases}\] 
\end{proposition}

We use Propositions~\ref{prop:disjoint_union} and~\ref{prop:lexicographicproductgraphdistance} to provide a characterization of the $k$-independence number of lexicographic graph products for $k \geq 2$. Let $\iota(G)$ denote the number of isolated vertices of graph $G$, and let $\mathcal{C}(G)$ denote the set of connected components of graph $G$.

\begin{theorem}
	\label{thm:product_graphs_lexicographic}
	For all graphs $G_1, G_2$ and all values $k \geq 2$, $$\alpha_k(G_1 \cdot G_2) = \alpha_k(G_1) + \iota(G_1) (\alpha_k(G_2) - 1).$$ 
\end{theorem}
\begin{proof}
	Let $\mathcal{C}(G_1)$ be the collection of connected components of graph $G_1$. Moreover, let $\mathcal{C}(G_1)$ be partitioned by the sets $\mathcal{C}_1$ and $\mathcal{C}_2$, where $\mathcal{C}_1$ is the set of all isolated vertices of graph $G_1$, and $\mathcal{C}_2$ is the set of all connected components of graph $G_1$ containing at least two vertices. We note that product graph $G_1 \cdot G_2$ is the disjoint union of the graphs $C \cdot G_2$ for $C \in \mathcal{C}(G_1)$. Then, by Proposition~\ref{prop:disjoint_union}, it follows that 
	\begin{align*}
	\alpha_k(G_1\cdot G_2) &= \sum_{C \in \mathcal{C}(G_1)} \alpha_k(C \cdot G_2)= \sum_{C \in \mathcal{C}_1} \alpha_k(C \cdot G_2) + \sum_{C \in \mathcal{C}_2} \alpha_k(C \cdot G_2).
	\end{align*}
	
	Consider a component $C \in \mathcal{C}_1$. By the definition of set $\mathcal{C}_1$, component $C$ consists of a single vertex. Therefore, $C \cdot G_2 = G_2$, and thus $\alpha_k(C \cdot G_2) = \alpha_k(G_2)$. Moreover, by the definition of set $\mathcal{C}_1$, $|\mathcal{C}_1| = \iota(G_1)$. It then follows that $$\sum_{C \in \mathcal{C}_1} \alpha_k(C \cdot G_2) = \iota(G_1) \cdot \alpha_k(G_2).$$
	
	Next, consider a component $C \in \mathcal{C}_2$. By the definition of set $\mathcal{C}_2$, component $C$ is connected and contains at least two vertices. 
	
	Let $S \subseteq V(G_1) \times V(G_2)$ be a maximum $k$-independent set in the graph $C \cdot G_2$. We observe that set $S$ does not contain a pair of vertices $(u, v_1), (u, v_2)$ such that $v_1 \neq v_2$. Namely, as component $C$ is connected, and as $|V(C)| \geq 2$, vertex $u$ has degree at least one. Hence, by Proposition~\ref{prop:lexicographicproductgraphdistance}, $\delta_{C \cdot G_2}((u, v_1), (u, v_2)) \leq 2 \leq k$. Thus, by the definition of $k$-independent sets, set $S$ contains at most one vertex with vertex $u$ as the first coordinate for all vertices $u \in V(C)$. 
	
	Let set $S' \subseteq V(C)$ be defined as $S' = \{u \in V(C) \, | \, (u,v) \in S\}$. Thus, $|S'| = |S|$. We moreover note that by Proposition~\ref{prop:lexicographicproductgraphdistance} and as $k \geq 2$, set $S$ is $k$-independent in product graph $C \cdot G_2$ if and only if set $S'$ is $k$-independent in component $C$. Thus, $\alpha_k(C) \geq \alpha_k(C \cdot G_2)$. 
	
	Next, let $S^* \subseteq V(C)$ be a maximum $k$-independent set in component $C$. Furthermore, let $v \in V(G_2)$ be an arbitrary vertex in graph $G_2$. Then, by Proposition~\ref{prop:lexicographicproductgraphdistance} and as $k \geq 2$, it follows that the set $S^* \times \{v\} \subseteq V(C) \times V(G_2)$ forms a $k$-independent set in product graph $C \cdot G_2$. Because $|S^* \times \{v\}| = |S^*|$, we find that $\alpha_k(C \cdot G_2) \geq \alpha_k(C)$.
	
	Thus, $\alpha_k(C \cdot G_2) = \alpha_k(C)$. We note that by Proposition~\ref{prop:disjoint_union}, 
	\begin{align*}
	\alpha_k(G_1) = \sum_{C \in \mathcal{C}(G_1)}\alpha_k(C)= \sum_{C \in \mathcal{C}_1}\alpha_k(C) + \sum_{C \in \mathcal{C}_2}\alpha_k(C).
	\end{align*}
	Because each component $C \in \mathcal{C}_1$ consists of a single vertex by the definition of set $\mathcal{C}_1$, it follows that $\alpha_k(C) = 1$ for each $C \in \mathcal{C}_1$. Hence, and due to $|\mathcal{C}_1| = \iota(G_1)$, $$\sum_{C \in \mathcal{C}_2}\alpha_k(C) = \alpha_k(G_1) - \iota(G_1).$$ Therefore, 
	\begin{align*}
	\sum_{C \in \mathcal{C}_2}\alpha_k(C \cdot G_2) = \sum_{C \in \mathcal{C}_2}\alpha_k(C)= \alpha_k(G_1) - \iota(G_1),
	\end{align*}
	and thus,
	\begin{align*}
	\alpha_k(G_1 \cdot G_2) &= \sum_{C \in \mathcal{C}_1}\alpha_k(C \cdot G_2) + \sum_{C \in \mathcal{C}_2}\alpha_k(C \cdot G_2)\\
	&= \iota(G_1) \cdot \alpha_k(G_2) + \alpha_k(G_1) - \iota(G_1)\\
	&= \alpha_k(G_1) + \iota(G_1) (\alpha_k(G_2) - 1),
	\end{align*}
	as desired.
\end{proof}

\section{Concluding remarks}

We conclude by observing a relationship between the $k$-independence numbers of the four considered graph products.

\begin{remark}
	\label{rem:product_graphs_relation}
	For all graphs $G_1, G_2$, $$\alpha_k(G_1 \cdot G_2) \leq \alpha_k(G_1 \boxtimes G_2) \leq \alpha_k(G_1 \square G_2),\, \alpha_k(G_1 \times G_2).$$
\end{remark}

\begin{proof}
	By the definitions of the graph products, it holds that \[E(G_1 \square G_2), E(G_1 \times G_2) \subseteq E(G_1 \boxtimes G_2) \subseteq E(G_1 \cdot G_2),\] from which the result directly follows.
\end{proof}

Remark~\ref{rem:product_graphs_relation} extends \cite[Theorem~2.6]{Jha1994IndependenceGraphs}.

Note that as a result of Remark~\ref{rem:product_graphs_relation}, also Theorem~\ref{thm:product_graphs_strong_lower_bound} is tight for $G_1, G_2$ both complete graphs and $k > 1$. We additionally observe that the inequality $\alpha(G_1 \square G_2) \leq \alpha(G_1 \times G_2)$ in~\cite[Theorem~2.6]{Jha1994IndependenceGraphs} does not extend, as $\alpha_2(P_3 \square K_{1,3}) > \alpha_2(P_3 \times K_{1,3})$.

The \emph{distance-$k$ chromatic number}, which is the chromatic number of $G^k$, has received quite some attention since its introduction by Alon and Mohar \cite{Alon1993ThePowers}, see for instance the work by Kang and Pirot~\cite{KP2016,KP2018}. Our upper bounds on the $k$-independence number of graph products directly yield lower bounds on the corresponding distance-$k$ chromatic number. 

\subsection*{Acknowledgements}
The research of Aida Abiad has been partially supported by the FWO grant 1285921N.

\bibliographystyle{acm}
\bibliography{bibliography.bib}

\end{document}